\documentclass[leqno,11pt]{amsart}

\newcommand{\Zz}{\mathbb{Z}}

\newcommand{\Pp}{\mathbb{P}}

\newcommand{\rk}{\operatorname{rank}}

\newcommand{\Oo}{\mathcal{O}}

\newcommand{\Ll}{\mathcal{L}}

\usepackage[letterpaper,margin=1in]{geometry}

\usepackage[all]{xy} 
\usepackage{amsmath, amssymb, amsfonts, latexsym, mdwlist, amsthm, amscd}
\usepackage{subfig}
\usepackage{graphicx}
\usepackage{wrapfig}

\usepackage[bookmarks, colorlinks, breaklinks, pdftitle={Ulrich bundles on Enriques surfaces},
pdfauthor={Howard Nuer}]{hyperref}
\hypersetup{linkcolor=blue,citecolor=blue,filecolor=black,urlcolor=blue}

\usepackage{cite}

\usepackage{tikz}
\usetikzlibrary{calc,trees,positioning,arrows,chains,shapes.geometric,%
    decorations.pathreplacing,decorations.pathmorphing,shapes,%
    matrix,shapes.symbols}

\tikzset{
>=stealth',
  punktchain/.style={
    rectangle,
    rounded corners,
    draw=black, thick,
    minimum height=3em,
    text centered,
    on chain},
  line/.style={draw, thick, <-},
  element/.style={
    tape,
    top color=white,
    bottom color=blue!50!black!60!,
    minimum width=8em,
    draw=blue!40!black!90, very thick,
    text width=10em,
    minimum height=3.5em,
    text centered,
    on chain},
  every join/.style={->, thick,shorten >=1pt},
  decoration={brace},
  tuborg/.style={decorate},
  tubnode/.style={midway, right=2pt},
}

\usepackage{paralist}
\setdefaultenum{(a)}{(i)}{}{}
\usepackage{enumitem} 

\def\deg{\mathop{\mathrm{deg}}}

\def\dim{\mathop{\mathrm{dim}}\nolimits}

\def\Ext{\mathop{\mathrm{Ext}}\nolimits}
\def\ext{\mathop{\mathrm{ext}}\nolimits}

\def\Hilb{\mathop{\mathrm{Hilb}}\nolimits}

\def\Num{\mathop{\mathrm{Num}}\nolimits}

\def\rk{\mathop{\mathrm{rk}}}

\def\MG13{\ensuremath{{\mathcal M}_{\Gamma_1(3)}}}
\def\tildeMG13{\ensuremath{\widetilde{\mathcal M}_{\Gamma_1(3)}}}

\def\into{\ensuremath{\hookrightarrow}}

\makeatletter
\newtheorem*{rep@theorem}{\rep@title}
\newcommand{\newreptheorem}[2]{%
\newenvironment{rep#1}[1]{%
 \def\rep@title{#2 \ref{##1}}%
 \begin{rep@theorem}}%
 {\end{rep@theorem}}}
\makeatother

\newtheorem{Thm}{Theorem}[section]
\newreptheorem{Thm}{Theorem}
\newtheorem{Prop}[Thm]{Proposition}

\newtheorem{Lem}[Thm]{Lemma}

\newtheorem{Cor}[Thm]{Corollary}
\newreptheorem{Cor}{Corollary}
\newtheorem{Con}[Thm]{Conjecture}
\newreptheorem{Con}{Conjecture}

\newtheorem{thm-int}{Theorem}

\theoremstyle{definition}
\newtheorem{Def-s}[Thm]{Definition}

\newtheorem{Rem}[Thm]{Remark}

\def\P{\ensuremath{\mathbb{P}}}

\def\Z{\ensuremath{\mathbb{Z}}}

\def\F{\ensuremath{\Delta}}

\def\CC{\ensuremath{\mathcal C}}
\def\DD{\ensuremath{\mathcal D}}
\def\EE{\ensuremath{\mathcal E}}
\def\FF{\ensuremath{\mathcal F}}
\def\GG{\ensuremath{\mathcal G}}

\def\LL{\ensuremath{\mathcal L}}
\def\MM{\ensuremath{\mathcal M}}

\def\OO{\ensuremath{\mathcal O}}

\newcommand{\ignore}[1]{}

\begin{document}

\title[Ulrich Bundles on Enriques Surfaces]{Ulrich Bundles on Enriques Surfaces}

\author{Lev Borisov}
\address{Department of Mathematics, Rutgers University, 110 Frelinghuysen Rd., Piscataway, NJ 08854, USA}
\email{borisov@math.rutgers.edu}
\urladdr{http://math.rutgers.edu/~borisov/}

\author{Howard Nuer}
\address{Department of Mathematics, Rutgers University, 110 Frelinghuysen Rd., Piscataway, NJ 08854, USA}
\email{hjn11@math.rutgers.edu}
\urladdr{http://math.rutgers.edu/~hjn11/}

\keywords{
}

\subjclass[2010]{}

\begin{abstract}We study Ulrich bundles and their moduli on unnodal Enriques surfaces.  In particular, we prove that unnodal Enriques surfaces are of wild representation type by constructing moduli spaces of stable Ulrich bundles of arbitrary rank and arbitrarily large dimension.
\end{abstract}

\vspace{-1em}

\maketitle

\tableofcontents
\section{Introduction}
Since Horrocks proved his seminal result \cite{Hor64} that ACM (Arithmetically Cohen-Macaulay) bundles on $\P^n$ are direct sums of line bundles, ACM sheaves on a projective variety have been intensely studied.  In addition to having important connections to the study of deformation theory (for example, to codimension 2 subschemes of $\P^n$ with ACM structure sheaves \cite{Ell75}), the study of ACM sheaves provides an important tool for determining the complexity of the underlying variety in terms of the dimension and number of families of indecomposable ACM sheaves that it supports.  This complexity is called the representation type of the variety \cite{MR15}.  

\medskip
Varieties admitting only a finite number of indecomposable ACM sheaves (up to twist and isomorphism) are called of \emph{finite representation type} and have been completely classified \cite{Yos90} : $[Z]\in\Hilb^n(\P^2)$ with $Z$ reduced and $n\leq 3$, $\P^n$, a smooth quadric hypersurface $X\subset\P^n$, a cubic scroll in $\P^4$, the Veronese surface in $\P^5$, or a rational normal curve.  At the other extreme are varieties of \emph{wild representation type} for which there exist families of non-isomorphic indecomposable ACM sheaves of arbitrarily large dimension.  Varieties of \emph{tame representation type} fit between these two extremes and for each rank $r$ admit only finitely many moduli spaces of indecomposable ACM sheaves of rank $r$, whose dimensions do not exceed one.

\medskip
Smooth projective curves fit perfectly into this trichotomy based on the genus $g$:

\begin{table}[ht]
\begin{tabular}{|c| c| c|}
\hline
$g=0$ & $g=1$ & $g\geq 2$\\
\hline
Finite & Tame & Wild\\
\hline
\end{tabular}
\end{table}

In higher dimension we cannot expect such a simple trichotomy to hold.  Indeed, a quadric cone in $\P^3$ exhibits an infinite discrete set indecomposable ACM bundles of rank 2 \cite{CH04}.  Nevertheless, the definition of varieties of wild representation type continues to make sense in higher dimension although few examples are known.  

\medskip
In this paper we study a particular kind of ACM bundle whose associated graded module has the maximal number of generators.  Although the algebraic counterpart of this phenomenon was first observed by Ulrich in \cite{Ulr84}, these so-called Ulrich sheaves on projective varieties were originally introduced in \cite{ES03}.  On a $d$-dimensional projective variety $X\subset\P^n$, an Ulrich sheaf $E$ can equivalently be described as a coherent $\OO_X$-module admitting a linear resolution as an $\OO_{\P^n}$-module, or as an $\OO_X$-module such that $\pi_*(E)\cong\OO_{\P^d}^{\deg(X)\rk(E)}$ for a generic projection $\pi:X\to\P^d$.  The most convenient definition for us, however, is that $E(-i)$ is acyclic for $i=1,\dots,d$, i.e. $H^j(E(-i))=0$ for $i=1,\dots,d$ and all $j$.  These equivalences are the content of \cite[Proposition 2.1]{ES03}.

\medskip
Ulrich sheaves were originally studied in connection with Chow forms of subvarieties of $\P^n$ \cite{ES03}, but since their introduction they have been shown to be intimately connected to the minimal resolution conjecture (MRC) \cite{AFO12,FMP03,EPSW02}, representations of Clifford algebras \cite{CKM12}, and Boij-S\"{o}derberg Theory \cite{ES11}.  Indeed, in \cite{ES11} it is shown that the \emph{cohomology table} $\CC(X,\OO_X(1))$ of a $d$-dimensional projective variety $X\subset \P^n$ is the same as that of $\P^d$ if and only if $X$ admits an Ulrich sheaf.  Furthermore, Ulrich sheaves are extremal rays in $\CC(X,\OO_X(1))$.

\medskip
Although Ulrich sheaves are conjectured to exist on every projective variety \cite{ES03}, Ulrich bundles have only been constructed in a few cases: smooth curves \cite{ES03}, complete intersections \cite{HUB91}, Grassmannians \cite{CMR15} and some partial Flag varieties \cite{CCHMRW15}, del Pezzo surfaces \cite[Corollary 6.5]{ES03} and more general rational surfaces with an anticanonical pencil \cite{Kim16}, K3 surfaces \cite{AFO12}, and abelian surfaces \cite{Bea15}.  In this short note we will study Ulrich bundles on a generic (more specifically: unnodal) Enriques surface with respect to multiples of the Fano polarization.

\medskip
In Section \ref{sec: Ulrich line bundles} we prove that Ulrich line bundles exist for any polarization proportional to the Fano polarization $\Delta$ of an unnodal Enriques surface $Y$. We also formulate an attractive lattice theoretic conjecture that would ensure that Ulrich line bundles exist for all polarizations.

\medskip
In Section \ref{sec: moduli} we prove that Enriques surfaces are of wild representation type by constructing stable Ulrich bundles of any rank, whose moduli spaces have increasingly large dimension.

\medskip
By using a particularly convenient resolution of the cotangent bundle $\Omega_Y$ that comes from the description of an unnodal Enriques surface $Y$ as the \'{e}tale quotient of a $(2,2,2)$ divisor in $(\P^1)^3$, we prove in Section \ref{sec: rank 2} that $\Omega_Y(3\Delta)$ is a stable Ulrich bundle of rank two with respect to the polarization $H=2\Delta$.  We include a proof of this torically motivated description of $Y$ in an appendix (Section \ref{app: toric}).

\medskip
{\bf Acknowledgements.} L.B. has been partially supported by the  NSF grant DMS-1201466.

\section{Ulrich line bundles on unnodal Enriques surfaces}\label{sec: Ulrich line bundles}
Let $X$ be smooth projective surface.  Recall that with respect to a very ample polarization $H$, a vector bundle $\EE$ on $X$ is called \emph{Ulrich}
if and only if 
\begin{equation}\label{u}
H^i(X,\EE(-H)) = H^i(X,\EE(-2H))=0
\end{equation}
for all $i$. 

\medskip
The following lemma describes a sufficient and necessary condition for a line bundle $\Oo(D)$ on an unnodal\footnote{An Enriques surface is said to be \emph{unnodal} if it contains no smooth rational curves and \emph{nodal} otherwise.} Enriques surface to
be acyclic, in the sense that all of its cohomology groups vanish.
\begin{Prop}\label{acyclic} 
Let $Y$ be an unnodal Enriques surface.
The condition $H^i(Y,\Oo(D))=0$ for all $i$ is equivalent to $D^2=-2$.
\end{Prop}

\begin{proof}
The vanishing of cohomology implies that $\chi(\Oo(D))=0$ and thus $D^2=-2$ by the Riemann-Roch theorem. In the other direction, $D^2=-2$ implies $\chi(\OO(D))=0$. Thus, it suffices to show $h^0(D)=h^2(D)=0$. If $h^0(D)>0$, then there is a smooth rational curve on $Y$, contradicting unnodality.  Indeed, let $C=\sum m_i R_i$, with $m_i>0$ and $R_i$ integral, be the decomposition of an effective curve $C\in |D|$ into irreducible components.  Then as $R_i.R_j\geq 0$ for $i\neq j$, we conclude from $C^2=-2$ that there must be some index $i$ for which $R_i^2<0$.  It follows that $0\leq p_a(R_i)=\frac{1}{2}R_i^2+1<1$, i.e. $R_i$ is a smooth rational curve (with $R_i^2=-2$).  The statement for $h^2(D)=h^0(K_Y-D)$ is identical.
\end{proof}

In the case of a line bundle on an unnodal Enriques surface $Y$, Proposition \ref{acyclic} allows us to reformulate the general Ulrich bundle condition \eqref{u} strictly in terms of the intersection pairing on $\Num(Y)$. Specifically, for $\EE=\Oo(D)$, the Ulrich condition is equivalent to the existence of $D_1,D_2\in\Num(Y)$ such that
\begin{equation}\label{ulrich}
D_1^2=D_2^2=-2, ~D_1-D_2=H
\end{equation}
where $D_1=D-H$ and $D_2=D-2H$.

\medskip
The following purely lattice theoretic conjecture is appealing but we cannot presently prove it.
Recall that up to torsion the cohomology of $Y$ with the intersection pairing is the even unimodular 
lattice $\Lambda$ of signature $(1,9)$ given by $U\oplus {\mathbb E}_8(-1)$ where $U$ is the hyperbolic plane.
\begin{Con}\label{lat}
Any element $H\in \Lambda$ can be written as a difference of elements $D_1$ and $D_2$ of self-intersection $(-2)$.
\end{Con}

\medskip
The particular case of Conjecture \ref{lat} when $H$ is a class of a very ample line bundle would imply the existence of Ulrich line bundles on unnodal Enriques surfaces with arbitrary polarization. We believe 
the conjecture is true in general, having checked it on a number of small examples. In what follows, we will prove the conjecture for $H$ a multiple of the Fano polarization $\Delta$.

\medskip
To describe the Fano polarization $\Delta$, we recall a presentation of $\Lambda$ from \cite[Proposition 2.5.5]{CD89}
$$
\Lambda =\left \{\sum_{i=1}^{10} a_i E_i\mid~a_i\in \frac 13 \Zz, a_i-a_j \in \Zz\right\}
$$
with the pairings $E_i.E_j=1-\delta_{ij}$.
The Fano class $\Delta$ is given by 
$$
\Delta=\frac 13\sum_{i=1}^{10} E_i.
$$
Together with $E_i$, the element $\Delta$ generates $\Lambda$. It satisfies
$$
\Delta^2=10, ~\Delta.E_i = 3 \mbox{~for all $i$.}
$$

\begin{Rem}
The importance of $\Delta$ is highlighted by the fact that it gives the embedding of $Y$ of the lowest possible degree ($10$) and in the smallest projective space ($\P^5$) \cite{CD85}.
\end{Rem}

The main result of this section is the following.
\begin{Thm}\label{Ulrich line bundle}
For any $k\in \Zz$ there exist elements $D_1$ and $D_2$ of $\Lambda$ such that 
$D_1^2=D_2^2=-2$ and $D_1-D_2=k\Delta$.  Consequently, there exist Ulrich line bundles with respect to $k\Delta$ for any $k\geq 1$.
\end{Thm}

\begin{proof}
It is clear that the case $k<0$ can be reduced to $k>0$ by switching $D_1$ and $D_2$.  

\medskip
For even $k\geq 0$ and $c_1,\dots, c_4\in\Z$, we consider
$$D_1 = \frac 12 k\F + c_1 (E_1-E_2) + c_2(E_3-E_4) + c_3(E_5-E_6) + c_4 (E_7-E_8),\mbox{ and }$$
$$D_2 = -\frac 12 k\F+ c_1 (E_1-E_2) + c_2(E_3-E_4) + c_3(E_5-E_6) + c_4 (E_7-E_8).
$$
It is clear that $D_1-D_2=k\F$ and that
$$D_1^2= D_2^2 = \frac 52 k^2 -2(c_1^2+c_2^2+c_3^2+c_4^2),$$ so by Lagrange's four-square
theorem we can pick $c_1,\dots, c_4$ to make $D_i^2=-2$.

\medskip
To deal with odd $k\geq 3$, consider
 the divisor class $L = 3E_{9}+2E_{10}-\F$.  We have $L.\F  = 15-10  = 5$, and 
$L^2 = 12 - 30 + 10 = -8$.
We will look for solutions to $D_1^2=D_2^2=-2, D_1-D_2=k\F$ of the form
$$
D_1 = \frac{k+1}{2}\F - L + c_1(E_1-E_2) +\cdots+c_4(E_7-E_8),~
D_2 = -\frac{k-1}{2}\F -L + c_1(E_1-E_2) + \cdots + c_4(E_7-E_8).$$
We have $D_1-D_2 = k\F$, and note that $(E_1-E_2),\ldots,(E_7-E_8)$ are mutually orthogonal, as well as orthogonal to $\F$ and $L$. The conditions $D_1^2=-2$ and $D_2^2=-2$ translate into
$$-2 = \frac 52(k+1)^2 - 5(k+1) -8 -2(c_1^2+c_2^2+c_3^2+c_4^2)$$
$$-2 = \frac 52(k-1)^2 +5(k-1) -8 - 2(c_1^2+c_2^2+c_3^2+c_4^2)$$
which are both equivalent to
$$c_1^2+c_2^2+c_3^2+c_4^2=\frac 14(5k^2-17).$$
The right hand side is a positive integer which can again be written
as a sum of four squares by Lagrange's theorem.

\medskip
It remains to consider the case $k=1$.
Here we provide a solution explicitly as 
$$
D_1=\frac 53 E_1 + \frac 23 \sum_{i=2}^4 E_i - \frac  13\sum_{i=5}^{10}E_i,
~
D_2=\frac 43 E_1 + \frac 13 \sum_{i=2}^4 E_i - \frac  23\sum_{i=5}^{10}E_i.
$$
\end{proof}

\begin{Rem}\label{rem: finite}
It is easy to show that for any $H$ with $H^2>0$ the number of solutions to the Ulrich equations \eqref{ulrich} is finite. Indeed, the sum $F=D_1+D_2$ lies in the orthogonal complement of $H$, which is 
a negative definite lattice. It satisfies $F^2= 2D_1^2+2D_2^2 - (D_1-D_2)^2 = -8-H^2$, which has only a finite number of solutions. Since $D_1-D_2=H$, we have only a finite number of pairs
$(D_1,D_2)=(\frac{F+H}{2},\frac{F-H}{2})$. In the case $H=\Delta$ and $H=2\Delta$ these calculations can be performed by hand. 
For $H=\F$, the resulting Ulrich divisors $D=D_1+H$, up to permutations of the $E_i$, are
$$
D=2E_1+E_2+E_3+E_4 ~\mbox{or~} D= E_1+\cdots +E_6-E_7.
$$
For $H=2\F$, all possible $D=D_1+H$ up to permutations are $\sum_i a_iE_i$ with $(a_1,\ldots,a_{10})$ in 
the set
\begin{align*}
 \{
(4,1,1,1,1,1,1,0,0,0), (3,3,1,1,1,1,0,0,0,0),(3,2,2,2,1,0,0,0,0,0),(3,2,2,1,1,1,1,0,0,-1),
\\
(2,2,2,2,2,1,0,0,0,-1),(2,2,2,2,1,1,1,1,-1,-1),(2,2,2,1,1,1,1,1,1,-2)
\}.
\end{align*}
We leave the details to the reader.
\end{Rem}

\section{Higher rank Ulrich bundles and their moduli}\label{sec: moduli}
In this section we prove the existence of stable Ulrich bundles of arbitrary rank with respect to the Fano polarization $H=\F$ on an unnodal Enriques surface $Y$.  As these bundles move in moduli spaces of arbitrarily large dimension, it will follow that unnodal Enriques surfaces are of wild representation type.  Before we get to the construction, we prove a higher rank finiteness result in analogy with Remark \ref{rem: finite}.   

\begin{Lem}  For any polarization $H$ and rank $r$, there exist only finitely many pairs $(c_1,c_2)$ for the Chern classes of vector bundle $E$ that is Ulrich with respect to $H$.
\end{Lem}
\begin{proof}  Riemann-Roch and the vanishing conditions $\chi(E(-H))=\chi(E(-2H))=0$ together imply that \begin{equation}\label{Ulrich conditions}(A)~~ c_1(E).H=\frac{3r}{2}H^2\mbox{ and }(B)~~c_2(E)=\frac{1}{2}c_1(E)^2-(H^2-1)r,
\end{equation}
so in particular $c_2(E)$ is determined by $r$ and $c_1(E)$.  From $(A)$ we see that the divisor class $2c_1(E)-3rH$ is orthogonal to $H$ and thus sits in the orthogonal complement of $H$, which is negative definite by the Hodge index theorem.  Furthermore, from the proof of \cite[Proposition 3.1(a)]{CH12}, it follows that $c_2(E)$ represents an effective $0$-dimensional cycle, so $c_2(E)\geq 0$.

\medskip
Solving for $c_1(E)^2$ in $(B)$ and applying the positivity of $c_2(E)$, we get that 
$$
0\geq(2c_1(E)-3rH)^2=4c_1(E)^2-9r^2 H^2\geq -8r+(8r-9r^2)H^2,
$$
where the rightmost bound depends only on $r$ and $H$.  Finiteness easily follows.
\end{proof}

Now we can construct higher-rank bundles.

\begin{Thm}\label{higher rank}
Let $Y$ be an unnodal Enriques surface.  For any $r\geq 1$, there exist rank $r$ stable Ulrich bundles on $Y$ with respect to $H=\F$.  Moreover, the moduli space $M_{H,Y}^U(r,c_1)$ of Ulrich bundles of rank $r$ and first Chern class $c_1$ has dimension $c_1^2-19r^2+1$.  
\end{Thm}
\begin{proof}  As noted above, we can certainly construct strictly semistable Ulrich bundles by taking iterated extensions of Ulrich line bundles.  To construct \emph{stable} Ulrich bundles, however, we must work a little harder and use some general machinery, though extensions will be our starting point.  

\medskip
By Theorem \ref{Ulrich line bundle}, we may assume that $r\geq 2$ and that by induction on $r$ we have constructed stable Ulrich bundles of smaller rank.  Take a stable Ulrich bundle $\FF$ of rank $r-1$ and first Chern class $c_1$ and an Ulrich line bundle $\OO(D)$ from Remark \ref{rem: finite} and consider extensions

\begin{equation}\label{ext}
0\to\OO(D)\to\EE\to\FF\to 0,
\end{equation} with the non-split ones parametrized by $\P\Ext^1(\FF,\GG)$.  While if $r>2$ $\FF$ and $\OO(D)$ are automatically not isomorphic, we may simply choose $\FF$ and $\OO(D)$ to be non-isomorphic in case $r=2$, so we assume this to be the case henceforth.  We may also choose $\OO(D)$ such that $\OO(D)\ncong\FF(K_Y)$.  

\medskip
It follows that $\ext^1(\FF,\OO(D))=-\chi(\FF,\OO(D))$ as $\hom(\FF,\OO(D))=0$, because $\FF$ and $\OO(D)$ are non-isomorphic stable sheaves of the same reduced Hilbert polynomial, and $\ext^2(\FF,\OO(D))=\hom(\OO(D),\FF(K_Y))=0$ for the same reason.  By \cite[Proposition 2.12]{CH12}, \begin{equation}\label{ext value}
\ext^1(\FF,\OO(D))=-\chi(\FF,\DD)=c_1(\FF).D-19\rk(\FF).\end{equation}  Since $\FF$ and $\OO(D)$ have the same slope, namely $15$, it follows from the Hodge index theorem that $(c_1(F)-\rk(\FF)D)^2\leq 0$.  By induction, $c_1(\FF)^2\geq 19\rk(\FF)^2-1$, so 
\begin{equation}\label{bound}c_1(\FF).D\geq\frac{c_1(\FF)^2}{2\rk(\FF)}+D^2\rk(\FF)=\frac{c_1(\FF)^2}{2\rk(\FF)}+18\rk(\FF).
\end{equation} Thus we get \begin{equation}\label{positive}
\ext^1(\FF,\OO(D))\geq\frac{17\rk(\FF)^2-1}{2\rk(\FF)}>0,\end{equation} so there exist non-trivial extensions $\EE$.  Since $\FF$ and $\OO(D)$ are non-isomorphic stable bundles of the same slope, such $\EE$ are necessarily \emph{simple}, i.e. $\hom(\EE,\EE)=1$ \cite[Lemma 4.2]{CH12}.  

\medskip
For simple bundles, we can use the existence of a modular family $\MM$ of simple bundles \cite[Proposition 2.10]{CH12} which parametrizes every simple bundle of the given Chern class at least once (but only finitely many times) and most importantly pro-represents the local deformation functor for each simple bundle.  In particular, every rank $r$ stable Ulrich bundle of first Chern class $c_1(\FF)+D$, if it exists, is represented by a point on $\MM$.  Furthermore, the dimension of $\MM$ is given by $$\ext^1(\EE,\EE)=-\chi(\EE,\EE)+1=c_1(\EE)^2-19r^2+1=(c_1(\FF)+D)^2-19r^2+1.$$

\medskip
If the generic point of the component of $\MM$ containing $\EE$ as in \eqref{ext} does not represent a stable sheaf, then the generic point represents strictly semistable Ulrich sheaves with Jordan-H\"{o}lder factors a rank $r-1$ stable Ulrich bundle and an Ulrich line bundle.  Indeed, both being Ulrich and stable are open in families, and by \cite[Proposition 2.3.1]{HL10} only such semistable splitting types can specialize to that of $\EE$.  

\medskip
On the other hand, the dimension those simple bundles $\EE$ coming from the construction described in \eqref{ext} is $$\dim M_{H,Y}^U(r-1,c_1(\FF))+\dim M_{H,Y}^U(1,D)+\dim \P\Ext^1(\FF,\OO(D))=c_1(\FF)^2+c_1(\FF).D-19r^2+19r.$$  That this is strictly smaller than $$\dim\MM=(c_1(\FF)+D)^2-19r^2+1=c_1(\FF)^2+2c_1(\FF).D-19r^2+19,$$ is precisely the positivity statement \eqref{positive} given the equality in \eqref{ext value} and the fact that $\rk(\FF)=r-1$.  Thus the general bundle represented by (this component of) $\MM$ is stable as well.  This concludes the proof that stable Ulrich bundles exist of arbitrary rank.  

\medskip
It only remains to describe the moduli space $M_{H,Y}^U(r,c_1)$ of Ulrich bundles of rank $r$ and first Chern class $c_1$ and its dimension.  To obtain a nice separated moduli space, we simply define $M_{H,Y}^U(r,c_1)$ to be the open locus of Ulrich bundles in the moduli space $M_{H,Y}(r,c_1)$, which parametrizes $S$-equivalence classes of Gieseker semistable sheaves.  Here, two semistable sheaves are called $S$-equivalent if they have the same graded object associated to a Jordan-H\"{o}lder filtration.  From \cite{Nue14a}, we know that locus of stable sheaves $M^s_{H,Y}(r,c_1)$ has the expected dimension $-\chi(\EE,\EE)+1$, and from above, this locus intersects $M_{H,Y}^U(r,c_1)$.  It follows that $\dim M_{H,Y}^U(r,c_1)=c_1^2-19r^2+1$ as claimed.
\end{proof}
\begin{Rem} A similar argument produces stable bundles of any rank that are Ulrich with respect to an arbitrary multiple of a Fano polarization.
\end{Rem}

We obtain the following result of independent interest as an immediate consequence:

\begin{Cor} Unnodal Enriques surfaces are of wild representation type.
\end{Cor}
\begin{proof} It suffices to demonstrate moduli spaces of stable Ulrich bundles of arbitrarily large dimension, as these contain open subsets parametrizing non-isomorphic stable (and thus indecomposable) ACM sheaves.  So write $r=2k+\epsilon$, where $\epsilon=0,1$, depending on the parity of $r$, and consider $$c_1=3k\Delta+\epsilon(E_7+E_8+E_9+2E_{10}).$$  As $3\Delta=(2E_1+E_2+E_3+E_4)+(-E_1+E_5+\dots+E_{10}),$ it follows that $c_1$ is in the semigroup generated by Ulrich divisors.  From the proof of Theorem \ref{higher rank}, there exist rank $r$ stable Ulrich bundles with first Chern class $c_1$.  These moduli spaces have dimension $$c_1^2-19r^2+1=14k^2+14k\epsilon-\epsilon^2+1,$$ which grows as $r$ grows, so the corollary follows.
\end{proof}
\section{A natural stable rank two Ulrich bundle}\label{sec: rank 2}
We showed above that there exists stable Ulrich bundles of arbitrary rank on $Y$ by taking extensions of Ulrich line bundles.  Here, we use a different method to prove that a certain natural rank two vector bundle on an unnodal Enriques surface $Y$ is Ulrich for the double Fano polarization $H=2\F$.

\begin{Thm}\label{Ulrich rank 2} Let $\F$ be a Fano polarization on an unnodal Enriques surface $Y$.  Then $\EE=\Omega_Y(3\F)$ is Ulrich with respect to the very ample divisor $H=2\F$, where $\Omega_Y$ is the cotagent bundle of $Y$.
\end{Thm}

\begin{proof}
In order to prove Theorem \ref{Ulrich rank 2}, we first recall a geometric construction of $Y$ as a free quotient of an invariant $(2,2,2)$ divisor in $(\Pp^1)^3$.  

\medskip
Let $E_1$, $E_2$ and $E_3$ be three half-pencils on $Y$ with $E_i.E_j=1-\delta_{i,j}$ in agreement with the notations of the previous section. Let $\pi:X\to Y$ be the K3 double cover of $Y$. For the pullbacks $\pi^*E_i$ on $X$, the linear systems $|\pi^*E_i|$ are base-point free pencils which give morphisms $X\to \Pp^1$ whose generic fibres are smooth elliptic curves. By combining these together we get a morphism
$$
f:X\to (\Pp^1)^3
$$
which is a closed embedding with image a $(2,2,2)$ divisor on $(\Pp^1)^3$, see the appendix, Section \ref{app: toric}, for details.
The line bundles $\pi^* E_i$ can be linearized with respect to the covering involution $\sigma:X\to X$.
This gives a linearization of $\Oo(1,1,1)$ on $(\Pp^1)^3$ and $X$ can be identified as the zero locus of a $\sigma$-invariant section of the $(2,2,2)$ line bundle on $(\Pp^1)^3$.

\medskip
The conormal exact sequence for $X\subset (\P^1)^3$, $$0\to \OO_X(-2,-2,-2)\to\Omega_{(\P^1)^3}|_X\to\Omega_X\to 0,$$ can be written 
$$
0\to\OO_X(-2,-2,-2)\to\OO_X(-2,0,0)\oplus\OO_X(0,-2,0)\oplus\OO_X(0,0,-2)\to\Omega_X\to 0,
$$
as $\Omega_{(\P^1)^3}\cong p_1^*\Omega_{\P^1}\oplus p_2^*\Omega_{\P^1}\oplus p_3^*\Omega_{\P^1}$.  Pushing forward and taking $\sigma$-invariant components, we get a resolution 
\begin{equation}\label{resolution2}0\to\OO_Y(-2E_1-2E_2-2E_3)\to\OO_Y(-2E_1)\oplus\OO_Y(-2E_2)\oplus\OO_Y(-2E_3)\to\Omega_Y\to 0
\end{equation}  
which will be the key to our verification of the Ulrich property.

\medskip
When we twist \eqref{resolution2}  by $\F$, we see that $\EE(-H)=\Omega_Y(\F)$ is resolved by $\displaystyle\bigoplus_{i=1}^3\OO_Y(\F-2E_i)$ and $\OO_Y(\F-2E_1-2E_2-2E_3)$.
Observe that $(\F-2E_i)^2=-2=(\F-2E_1-2E_2-2E_3)^2$. 
By Proposition \ref{acyclic}, each of the line bundles $\OO_Y(\F-2E_i),i=1,2,3$, and $\OO_Y(\F-2E_1-2E_2-E_3)$ has zero cohomology groups.  It follows that $\Omega_Y(\F)$ is acyclic. 

\medskip
To prove that $\EE$ is Ulrich it only remains to show that $\EE(-2H)=\Omega_Y(-\F)$ is acyclic.  From Serre duality and the fact that $\Omega_Y^\vee = \Omega_Y\otimes K_Y$, it follows that
$$
h^i(\Omega_Y(-\F))= h^{2-i}(K_Y\otimes \Omega_Y^\vee(\F)) = h^{2-i}(\Omega_Y(\F))=0,
$$
by the previous paragraph.
\end{proof}

\medskip
In the remainder of this section we will show that the Ulrich bundle $\Omega_Y(3\F)$ is $\mu$-stable 
with respect to the Fano polarization $\F$. 

\begin{Thm}\label{stable}
The Ulrich bundle $\Omega_Y(3\F)$ on an unnodal Enriques surface $Y$ is $\mu$-stable with respect 
to the polarization $\F$.
\end{Thm}

\begin{proof}
To show that $\Omega_Y(3\F)$ is $\mu$-stable, it suffices to prove that $\Omega_Y$ is $\mu$-stable.  Furthermore, as Ulrich bundles are always Gieseker semistable (and thus $\mu$-semistable) \cite[Theorem 2.9]{CH12}, we may assume that we have a rank 1 subsheaf $\Ll\subset\Omega_Y$ with $\mu_{\F}(\LL)=\mu_{\F}(\Omega_Y)=0$.  We may, of course, assume that $\Ll=\Oo(D)$ is a line bundle since the line bundle $\LL^{\vee\vee}$ has the same slope and is also a subsheaf of $\Omega_Y$.  Thus we are interested in the case $D.\F= 0$. 

\medskip
If we have an embedding $\Oo(D)\to \Omega_Y$, then we have a nonzero global section
of $\Omega_Y(-D)$. Since $h^0(\Omega_Y ) =h^0(\Omega^\vee_Y)=0$, we may assume that 
$D$ is not numerically trivial.
As discussed in the proof of Theorem \ref{Ulrich rank 2}, for any triple of distinct indices $\{i,j,k\}\subset \{1,\ldots,10\}$
we have a short exact sequence 
$$
0\to \Oo(-2E_i-2E_j-2E_k) \to \Oo(-2E_i)\oplus \Oo(-2E_j)\oplus \Oo(-2E_k) \to \Omega_Y \to 0
$$
which can be dualized and twisted by the canonical class to give
$$
0\to \Omega_Y \to  \Oo(2E_i+K)\oplus \Oo(2E_j+K)\oplus \Oo(2E_k+K)\to 
\Oo(2E_i+2E_j+2E_k+K)\to 0.
$$
Thus it suffices to show that there is a triple of indices $i,j,k$ such that 
\begin{equation}\label{all0}
h^0(2E_i+K-D)=h^0(2E_j+K-D)=h^0(2E_k+K-D)=0.
\end{equation}

\medskip
Suppose that there exists some $t$ such that $h^0(2E_t+K-D)>0$. By unnodality of $Y$ this implies 
$$
(2E_t-D)^2\geq 0.
$$
We know that 
$$
D=\sum_{i=1}^{10} a_i E_i +\mbox{torsion}
$$
with $a_i\in \frac 13\Zz, a_i-a_j\in \Zz$. From $0=\frac 13D.F = \sum_{i=1}^{10}a_i=0$
we conclude that $a_i\in \Zz$ for all $i$.

\medskip
The condition $(2E_t-D)^2\geq 0$ implies 
$$
0\leq -(2-a_t)\sum_{i\neq t}a_i + \sum_{i<j, i\neq t,j\neq t}a_ia_j =
-(2+s) s + \frac 12(s^2 - \sum_{i\neq t}a_i^2)
$$
where $s=\sum_{i\neq t}a_i$. We have used $a_t=-s$ from $D.F=0$. This 
translates into 
$$
(s+2)^2+\sum_{i\neq t} a_i^2  \leq 4.
$$
We can make two observations. First of all, there are at most $5$ nonzero $a_i$ (including $a_t$). Second, if $a_t=0$, then all other $a_i$ are zero, which contradicts $D$ being nontorsion. Now the first observation allows us to pick 
$(i,j,k)$ such that $a_i=a_j=a_k=0$ and then the second observation ensures \eqref{all0}.

\medskip
Therefore, $\Omega_Y$ is $\mu_{\F}$-stable.
\end{proof}

\section{Appendix: A toric projective model of $Y$}\label{app: toric}
In this section we show that the K3 cover $X$ of an Enriques surface $Y$ can be represented as a hypersurface in $\P^1\times\P^1\times\P^1$ of tridegree $(2,2,2)$.  

\medskip
To this end, let $E_1,E_2,E_3$ be three halfpencils on an Enriques surface $Y$ such that $E_i.E_j=1-\delta_{ij}$.  Denote their pull-backs under $\pi:X\to Y$ by $F_i=\pi^*E_i$.  Each $F_i$ induces an elliptic fibration $\pi_i:X\to\P^1$.  Indeed, $$h^i(X,\OO(F_i))=h^i(Y,\pi_*(\OO(F_i)))=h^i(E_i)+h^i(E_i+K_Y),$$ so $h^i(X,\OO(F_i))=0$ for $i>0$ and $h^0(X,\OO(F_i))=2$.  Basepoint-freeness follows from $F_i^2=0$.  Consider the map $\phi=\pi_1\times\pi_2\times\pi_3:X\to\P^1\times\P^1\times\P^1$.  

\medskip
Define $Q:=\phi(X)$, so that $Q$ is a hypersurface of $(\P^1)^3$.  We would like to show that $Q$ has tridegree $(2,2,2)$.  First notice that any curve contracted by $\phi$ would have to be one of the fibers of $\pi_i$ for some $i$.  But as $F_j.F_i=1\neq 0$, $\phi$ could never contract such a curve.  So $\phi$ is certainly a finite morphism.  We may, without loss of generality, consider a line of the form $\{p\}\times\{q\}\times\P^1$ for fixed $p,q\in\P^1$.  Then $\phi^{-1}(\{p\}\times\{q\}\times\P^1)=\pi_1^{-1}(p)\cap\pi_2^{-1}(q)$, which (generically) consists of two distinct points as $F_1.F_2=2$ \cite[Section 18]{BHPV}.  

\medskip
Suppose that for generic choices of $p$ and $q$, $Q\cap (\{p\}\times\{q\}\times\P^1)$ does not consist of two distinct points, but instead of a single point $(p,q,r)$ whose preimage under $\phi$ consists of two reduced points.  Then $F_3^r\cap F_1^p\cap F_2^q=F_1^p\cap F_2^q$ consists of two points, where we denote by $F_i^t$ the elliptic fibre of $|F_i|$ over the point $t\in\P^1$.  Consider the specific case where $F_1^p=\pi^*E_1$ and $F_2^q=\pi^*E_2$.  Then $F_1^p\cap F_2^q$ consists of two distinct points sitting over the unique point of intersection of $E_1$ and $E_2$.  By \cite[Proposition 2.2]{Ume15}, we may choose the half-pencils so that $E_1\cap E_2\cap(E_3+E_3')=\varnothing$, where $E_3'=E_3+K_Y$ is the adjoint half pencil.  It follows that the point $r\in\P^1$ above cannot be either of the points corresponding to the preimages of $E_3$ and $E_3'$.  From \cite[Remark after Lemma 17.3]{BHPV}, we see that other than the two double fibres $2E_3$ and $2E_3'$, whose preimages are irreducible, the preimages of the other curves in $|2E_3|$ consist of two disjoint, isomorphic curves which are switched by the covering involution of $\pi:X\to Y$.  Since $F_3^r\cap F_1^p\cap F_2^q=F_1^p\cap F_2^q$ and since $r$ must lie over a reduced fibre of $|2E_3|$, then we get an immediate contradiction as the two points of intersection must simultaneously have the same and distinct images in $Y$ under $\pi$.  Thus $(\{p\}\times\{q\}\times\P^1)\cap Q$ consists of two distinct points, and as this must then be the generic the case, it follows that $Q$ has tridegree $(2,2,2)$ so that $\phi$ is generically one-to-one.  It follows that $\phi$ is birational.  

\medskip
Now we notice that $H^0(Q,\OO(k,k,k))\into H^0(X,k(F_1+F_2+F_3))$ and both have the same dimension, namely $6k^2+2$, for all $k\geq 1$, so from the equality of Hilbert polynomials it follows that $\phi$ is an isomorphism onto its image, the $(2,2,2)$ divisor $Q$..  

\medskip
Keeping with the notation above, let $E_i':=E_i+K_Y$ be the adjoint halfpencil to $E_i$ and define $F_i':=\pi^*E_i'$.  Then we may choose sections $g_i,g_i'\in H^0(X,\OO_X(F_i)$ for $F_i$ and $F_i'$, respectively, such that the covering involution $\iota$ acts on these sections via $\iota^*(g_i)=g_i,\iota^*(g_i')=-g_i'$.  Consequently we can choose tri-homogeneous coordinates $([u_0:u_1],[v_0:v_1],[w_0:w_1])$ on $\P^1\times\P^1\times\P^1$ such that $\pi_i(x)=[g_i(x):g_i'(x)]$.  Then we define an involution $\tau$ on $(\P^1)^3$ by $\tau([u_0:u_1],[v_0:v_1],[w_0:w_1])=([u_0:-u_1],[v_0:-v_1],[w_0:-w_1])$ so that the embedding $\phi$ is $\Z/2\Z$-invariant for the actions of $\iota$ and $\tau$.  

\medskip
There are eight fixed points of $\tau$, \begin{align*}&([1:0],[1:0],[1:0]),([1:0],[1:0],[0:1]),([1:0],[0:1],[1:0]),([1:0],[0:1],[0:1]),\\&([0:1],[1:0],[1:0]),([0:1],[1:0],[0:1]),([0:1],[0:1],[1:0]),([0:1],[0:1],[0:1]).\end{align*}  As $X=Q$ has no fixed points, it cannot pass through these eight points and the defining $(2,2,2)$ equation of $Q$ must be invariant. 

\medskip
It is clear that conversely, any such $(2,2,2)$ $\tau$-invariant divisor whose vanishing locus $Q$ avoids the 8 fixed points of $\tau$ and is a smooth irreducible surface is the double cover of an Enriques surface $Y$.

Altogether this proves:
\begin{Thm}\label{represent} Ley $Y$ be an unnodal Enriques surface and $E_1,E_2,E_3$ be three halfpencils such that $E_i.E_j=1$ for $i\neq j$.  With $\iota$ and $\tau$ defined as above, there is a $\tau$-invariant trihomogeneous polynomial of tridegree $(2,2,2)$ in $[u_0:u_1],[v_0:v_1],[w_0:w_1]$, with zero-set $Q$ isomorphic to the universal K3 covering $X$ of $Y$ such that involution $\sigma$ is induced by the involution $\tau$ on $\P^1\times\P^1\times\P^1$.  The three rulings of $(\P^1)^3$ define the three elliptic pencils $|F_i|$ on $X$.  Conversely, any $\tau$-invariant trihomogenous polynomial of tridegree $(2,2,2)$ which avoids the fixed points of $\tau$ and defines a smooth irreducible surface $Q$ gives rise to the universal K3 cover of an Enriques surface $Y=Q/\tau$.  
\end{Thm}

Let us observe that the vector space of $\tau$-invariant trihomogeneous polynomials of tridegree $(2,2,2)$ is 14-dimensional and spanned by the eight monomials $u_0^{2i} u_1^{2-2i}v_0^{2j}v_1^{2-2j}w_0^{2k}w_1^{2-2k}$ for $i,j,k=0,1$ and the six monomials $u_0u_1v_0v_1w_0^{2k}w_1^{2-2k},u_0u_1v_0^{2j}v_1^{2-2j}w_0w_1,u_0^{2i}u_1^{2-2i}v_0v_1w_0w_1$.  This defines a linear system without base points, so by Bertini's the generic such $(2,2,2)$ divisor is smooth and irreducible.  Furthermore, the sublinear system consisting of $\tau$-invariant $(2,2,2)$ divisors passing through one of the fixed points of $\tau$ has codimension one, so the generic $\tau$-invariant $(2,2,2)$ divisor is smooth, irreducible, and avoids all of the eight fixed points.  

\medskip
It follows from Theorem \ref{represent} that any unnodal Enriques surface $Y$ can be embedded as a divisor in the toric variety $(\P^1)^3/\tau$.  While we have not seen the description here in the literature, it is intimately related to one of the most classical descriptions of Enriques surfaces as the normalization of sextic surfaces in $\P^3$ passing doubly through the coordinate tetrahedron \cite[p. 635]{GH94}, i.e. as the normalization of the hypersurface in $\P^3$ defined by $$x^2y^2z^2 +x^2y^2w^2+x^2z^2w^2+y^2z^2w^2+xyzwQ(x,y,z,w)=0.$$  Indeed, these \emph{Enriques representations}, as they were called in \cite{BP83}, correspond to the linear system $|E_1+E_2+E_3|$ whose pull-back defined $\phi:X\into (\P^1)^3$ above.  

\medskip
The picture below relates more directly the classical Enriques representation and the representation as a quotient of an invariant 
$(2,2,2)$ by an involution.

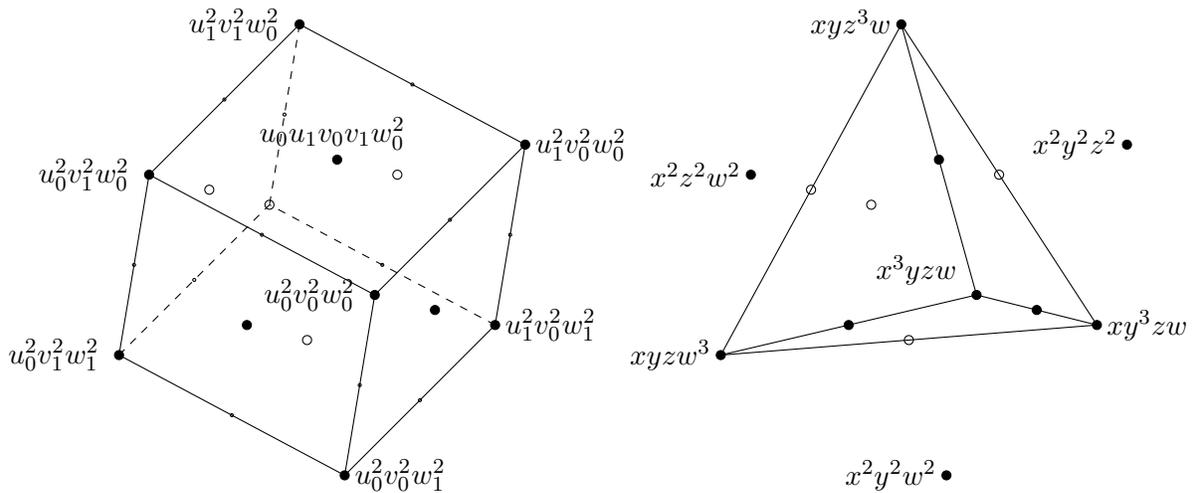
\[
\begin{tikzpicture}[scale=1]
\draw[fill]( 0,0) circle(.06);
\draw(1,1) circle(.02);
\draw[fill]( 2,2) circle (.06);
\draw(-1.5,.8) circle(.02);
\draw[fill]( -3,1.6) circle (.06);
\draw(-.2,-1.2) circle(.02);
\draw[fill]( -.4,-2.4) circle (.06);

\draw[fill]( -.4+2,-2.4+2) circle (.06);
\draw[fill]( -.4-3,-2.4+1.6) circle (.06);
\draw[fill]( 2-3,2+1.6) circle (.06);
\draw( 2-3-.4,2+1.6-2.4) circle (.06);
\draw[fill](-.2-1.5,-1.2+.8)circle(.06);
\draw[fill](1-1.5,1+.8)circle(.06);
\draw[fill](-.2+1,-1.2+1)circle(.06);
\draw(-.2+1-3,-1.2+1+1.6)circle(.06);
\draw(1-1.5-.4,1+.8-2.4)circle(.06);
\draw(-.2-1.5+2,-1.2+.8+2)circle(.06);

\draw(-.2+2,-1.2+2) circle(.02);
\draw(-.2-3,-1.2+1.6) circle(.02);
\draw(-.2-3+2,-1.2+1.6+2) circle(.02);
\draw(-1.5+2,.8+2) circle(.02);
\draw(-1.5-.4,.8-2.4) circle(.02);
\draw(-1.5-.4+2,.8-2.4+2) circle(.02);
\draw(1-3,1+1.6) circle(.02);
\draw(1-.4,1-2.4) circle(.02);
\draw(1-.4-3,1-2.4+1.6) circle(.02);

\draw (0,0) -- (2,2);
\draw (0,0) -- (-3,1.6);
\draw (0,0) -- (-.4,-2.4);
\draw (2-3,2+1.6) -- (2,2);
\draw (2-3,2+1.6) -- (-3,1.6);
\draw [dashed](2-3,2+1.6) -- (2-3-.4,2+1.6-2.4);
\draw (-.4+2,-2.4+2) -- (-.4,-2.4);
\draw (-.4-3,-2.4+1.6) -- (-.4,-2.4);
\draw (-.4-3,-2.4+1.6) -- (-3,1.6);
\draw[dashed] (-.4-3,-2.4+1.6) -- (-.4-3+2,-2.4+1.6+2);
\draw (-.4+2,-2.4+2) -- (2,2);
\draw[dashed] (-.4+2,-2.4+2) -- (-.4-3+2,-2.4+1.6+2);

\node[left] at (0,0) {$u_0^2v_0^2w_0^2~$};
\node[ right] at (2,2) {$u_1^2v_0^2w_0^2~$};
\node[ left] at (-3,1.6) {$u_0^2v_1^2w_0^2~$};
\node[right] at (-.4,-2.4) {$u_0^2v_0^2w_1^2~$};

\node[right] at ( -.4+2,-2.4+2) {$u_1^2v_0^2w_1^2~$};
\node[left] at  (-.4-3,-2.4+1.6) {$u_0^2v_1^2w_1^2~$};
\node[left] at  ( 2-3,2+1.6) {$u_1^2v_1^2w_0^2~$};

\node[above] at  (1-1.5,1+.8) {$u_0u_1v_0v_1w_0^2~$};

\draw[fill]( 8+0,0) circle(.06);
\draw[fill]( 8+2,2) circle (.06);
\draw[fill]( 8-3,1.6) circle (.06);
\draw[fill]( 8-.4,-2.4) circle (.06);

\draw[fill]( 8-.4+2,-2.4+2) circle (.06);
\draw[fill]( 8-.4-3,-2.4+1.6) circle (.06);
\draw[fill]( 8+2-3,2+1.6) circle (.06);
\draw( 8+2-3-.4,2+1.6-2.4) circle (.06);
\draw[fill](8-.2-1.5,-1.2+.8)circle(.06);
\draw[fill](8+1-1.5,1+.8)circle(.06);
\draw[fill](8-.2+1,-1.2+1)circle(.06);
\draw(8-.2+1-3,-1.2+1+1.6)circle(.06);
\draw(8+1-1.5-.4,1+.8-2.4)circle(.06);
\draw(8-.2-1.5+2,-1.2+.8+2)circle(.06);

\draw (8+0,0) -- (8+2-.4,2-2.4);
\draw (8+0,0) -- (8+2-3,2+1.6);
\draw (8+0,0) -- (8-.4-3,-2.4+1.6);
\draw (8+2-.4,2-2.4) -- (8-.4-3,-2.4+1.6);
\draw (8+2-.4,2-2.4) -- (8+2-3,2+1.6);
\draw (8-.4-3,-2.4+1.6) -- (8+2-3,2+1.6);

\node[above left] at (8 ,0) {$x^3yzw~$};
\node[right] at (8 -.4+2,-2.4+2) {$xy^3zw$};
\node[ left] at  (8 -.4-3,-2.4+1.6) {$xyzw^3$};
\node[left] at  ( 8+2-3,2+1.6) {$xyz^3w$};

\node[left] at (8+2,2){$x^2y^2z^2$};
\node[left] at (8-3,1.6){$x^2z^2w^2$};
\node[ left] at (8-.4,-2.4){$x^2y^2w^2$};
\end{tikzpicture}
\]

\smallskip
In the left panel we indicate invariant monomials
of tridegree $(2,2,2)$ on $(\Pp^1)^3$ with coordinates
$$
([u_0:u_1],[v_0:v_1],[w_0:w_1])
$$
with respect to the involution that fixes $u_0,v_0,w_0$ and negates
$u_1,v_1,w_1$.
There are $14$ such monomials that
can be thought of as vertices and midpoints of facets of a size two cube.
For clarity, only
some of the monomials are marked in the picture. The non-filled circles
indicate monomials
that are blocked from the view.

\smallskip
In the right panel, we write the same monomials as monomials on $\Pp^3$
with coordinates
$[x:y:z:w]$. There is a size two tetrahedron there whose monomials
correspond to those appearing in 
$xyzwQ(x,y,z,w)$ in the usual equation of Enriques surface
$$
x^2y^2z^2+x^2y^2w^2+x^2z^2w^2+y^2z^2w^2+xyzwQ(x,y,z,w) = 0.
$$
Clearly, one can make the four coefficients to be $1$ by scaling the
coordinates.
\bibliographystyle{plain}
\bibliography{NSF_Research_Proposal}

\begin{thebibliography}{10}

\bibitem{AFO12}
M.~Aprodu, F.~Gavril, and A.~Ortega.
\newblock Minimal resolutions, {C}how forms and {U}lrich bundles on {K3}
  surfaces.
\newblock {\em Crelle}, 12 2012.

\bibitem{BHPV}
W.~Barth, K.~Hulek, C.~Peters, and A.~Van de~Ven.
\newblock {\em Compact complex surfaces}, volume~4 of {\em Ergebnisse der
  Mathematik und ihrer Grenzgebiete. 3. Folge. A Series of Modern Surveys in
  Mathematics [Results in Mathematics and Related Areas. 3rd Series. A Series
  of Modern Surveys in Mathematics]}.
\newblock Springer-Verlag, Berlin, second edition, 2004.

\bibitem{BP83}
W.~Barth and C.~Peters.
\newblock Automorphisms of {E}nriques surfaces.
\newblock {\em Invent. Math.}, 73(3):383--411, 1983.

\bibitem{Bea15}
A.~Beauville.
\newblock Ulrich bundles on abelian surfaces.
\newblock 12 2015.

\bibitem{CH04}
M.~Casanellas and R.~Hartshorne.
\newblock Gorenstein biliaison and {ACM} sheaves.
\newblock {\em J. Algebra}, 278(1):314--341, 2004.

\bibitem{CH12}
M.~Casanellas, R.~Hartshorne with an appendix~by F.~Geiss, and F.-O. Schreyer.
\newblock Stable {U}lrich bundles.
\newblock {\em Internat. J. Math.}, 23(8):1250083, 50, 2012.

\bibitem{CKM12}
E.~Coskun, R.~Kulkarni, and Y.~Mustopa.
\newblock Pfaffian quartic surfaces and representations of {C}lifford algebras.
\newblock {\em Doc. Math.}, 17:1003--1028, 2012.

\bibitem{CCHMRW15}
I.~Coskun, L.~Costa, J.~Huizenga, R.~Mir{\'o}-Roig, and M.~Woolf.
\newblock Ulrich schur bundles on flag varieties.
\newblock 12 2015.

\bibitem{CD85}
F.~Cossec and I.~Dolgachev.
\newblock Smooth rational curves on {E}nriques surfaces.
\newblock {\em Math. Ann.}, 272(3):369--384, 1985.

\bibitem{CD89}
F.~Cossec and I.~Dolgachev.
\newblock {\em Enriques surfaces. {I}}, volume~76 of {\em Progress in
  Mathematics}.
\newblock Birkh{\"a}user Boston, Inc., Boston, MA, 1989.

\bibitem{CMR15}
L.~Costa and R.~Mir{{\'o}}-Roig.
\newblock {$GL(V)$}-invariant {U}lrich bundles on {G}rassmannians.
\newblock {\em Math. Ann.}, 361(1-2):443--457, 2015.

\bibitem{EPSW02}
D.~Eisenbud, S.~Popescu, F.-O. Schreyer, and C.~Walter.
\newblock Exterior algebra methods for the minimal resolution conjecture.
\newblock {\em Duke Math. J.}, 112(2):379--395, 2002.

\bibitem{ES11}
D.~Eisenbud and F.-O. Schreyer.
\newblock Boij-{S}{\"o}derberg theory.
\newblock In {\em Combinatorial aspects of commutative algebra and algebraic
  geometry}, volume~6 of {\em Abel Symp.}, pages 35--48. Springer, Berlin,
  2011.

\bibitem{ES03}
D.~Eisenbud and F.O. Schreyer.
\newblock Resultants and {C}how forms via exterior syzygies.
\newblock {\em J. Amer. Math. Soc.}, 16:537--579, 2003.

\bibitem{Ell75}
G.~Ellingsrud.
\newblock Sur le sch{\'e}ma de {H}ilbert des vari{\'e}t{\'e}s de codimension
  {$2$} dans {${\bf P}^{e}$} {\`a} c\^one de {C}ohen-{M}acaulay.
\newblock {\em Ann. Sci. {\'E}cole Norm. Sup. (4)}, 8(4):423--431, 1975.

\bibitem{FMP03}
G.~Farkas, M.~Musta\c{t}\v{a}, and M.~Popa.
\newblock Divisors on {${\MM}_{g,g+1}$} and the minimal resolution conjecture
  for points on canonical curves.
\newblock {\em Ann. Sci. {\'E}cole Norm. Sup. (4)}, 36(4):553--581, 2003.

\bibitem{GH94}
P.~Griffiths and J.~Harris.
\newblock {\em Principles of algebraic geometry}.
\newblock Wiley Classics Library. John Wiley \& Sons, Inc., New York, 1994.
\newblock Reprint of the 1978 original.

\bibitem{HUB91}
J.~Herzog, B.~Ulrich, and J.~Backelin.
\newblock Linear maximal {C}ohen-{M}acaulay modules over strict complete
  intersections.
\newblock {\em J. Pure Appl. Algebra}, 71(2-3):187--202, 1991.

\bibitem{Hor64}
G.~Horrocks.
\newblock Vector bundles on the punctured spectrum of a local ring.
\newblock {\em Proc. London Math. Soc. (3)}, 14:689--713, 1964.

\bibitem{HL10}
D.~Huybrechts and M.~Lehn.
\newblock {\em The geometry of moduli spaces of sheaves}.
\newblock Cambridge Mathematical Library. Cambridge University Press,
  Cambridge, second edition, 2010.

\bibitem{Kim16}
Y.~Kim.
\newblock Ulrich bundles on rational surfaces with an anticanonical pencil.
\newblock {\em Manuscripta Math.}, 150(1-2):99--110, 2016.

\bibitem{MR15}
R.~Mir{{\'o}}-Roig.
\newblock On the representation type of a projective variety.
\newblock {\em Proc. Amer. Math. Soc.}, 143(1):61--68, 2015.

\bibitem{Nue14a}
H.~Nuer.
\newblock A note on the existence of stable vector bundles on {E}nriques
  surfaces.
\newblock {\em Selecta Math., to appear. arXiv:math/1406.3328}, 2014.

\bibitem{Ulr84}
B.~Ulrich.
\newblock Gorenstein rings and modules with high numbers of generators.
\newblock {\em Math. Z.}, 188(1):23--32, 1984.

\bibitem{Ume15}
Y.~Umezu.
\newblock Normal quintic surfaces with {K}odaira dimension one.
\newblock {\em Internat. J. Math.}, 26(2):1550015, 22, 2015.

\bibitem{Yos90}
Y.~Yoshino.
\newblock {\em Cohen-{M}acaulay modules over {C}ohen-{M}acaulay rings}, volume
  146 of {\em London Mathematical Society Lecture Note Series}.
\newblock Cambridge University Press, Cambridge, 1990.

\end{thebibliography}
\end{document}